\documentclass[11pt]{article}
\usepackage{ulem}
\usepackage{graphicx,amsmath,amsfonts,amssymb,color,mathrsfs, amsmath,amsthm,bm}
\usepackage{epsfig}
\usepackage{ulem}
\makeatletter
 \newif\if@borderstar
 \def\bordermatrix{\@ifnextchar*{%
 \@borderstartrue\@bordermatrix@i}{\@borderstarfalse\@bordermatrix@i*}%
 }
 \def\@bordermatrix@i*{\@ifnextchar[{\@bordermatrix@ii}{\@bordermatrix@ii[()]}}
 \def\@bordermatrix@ii[#1]#2{%
 \begingroup
 \m@th\@tempdima8.75\p@\setbox\z@\vbox{%
 \def\cr{\crcr\noalign{\kern 2\p@\global\let\cr\endline }}%
 \ialign {$##$\hfil\kern 2\p@\kern\@tempdima & \thinspace %
 \hfil $##$\hfil && \quad\hfil $##$\hfil\crcr\omit\strut %
 \hfil\crcr\noalign{\kern -\baselineskip}#2\crcr\omit %
 \strut\cr}}%
 \setbox\tw@\vbox{\unvcopy\z@\global\setbox\@ne\lastbox}%
 \setbox\tw@\hbox{\unhbox\@ne\unskip\global\setbox\@ne\lastbox}%
 \setbox\tw@\hbox{%
 $\kern\wd\@ne\kern -\@tempdima\left\@firstoftwo#1%
 \if@borderstar\kern2pt\else\kern -\wd\@ne\fi%
 \global\setbox\@ne\vbox{\box\@ne\if@borderstar\else\kern 2\p@\fi}%
 \vcenter{\if@borderstar\else\kern -\ht\@ne\fi%
 \unvbox\z@\kern-\if@borderstar2\fi\baselineskip}%
 \if@borderstar\kern-2\@tempdima\kern2\p@\else\,\fi\right\@secondoftwo#1 $%
 }\null \;\vbox{\kern\ht\@ne\box\tw@}%
 \endgroup
 }

\newcommand{\chuhao}{\fontsize{19pt}{\baselineskip}\selectfont}

\numberwithin{equation}{section}

 \newtheorem{theorem}{Theorem}[section]
 \newtheorem{lemma}{Lemma}[section]

 \newtheorem{remark}{Remark}[section]

 \newtheorem{example}{Example}[section]
 
 \newtheorem{corollary}{Corollary}[section]
 
\newtheorem{algorithm}{Algorithm}[section]

\title{\bf\color{black} \chuhao{Matrix-analytic methods for solving  Poisson's equation  with applications to  Markov chains of $GI/G/1$-type}}
\author{Jinpeng Liu\thanks{School of Mathematics and Statistics, New Campus, Central
South University, Changsha, Hunan, 410083, P.R. China, E-mail:
liujinpeng@csu.edu.cn.
\newline \indent \hspace*{-2.5mm}
$^{**}$ School of Mathematics and Statistics, New Campus, Central
South University, Changsha, Hunan, 410083, P.R. China, E-mail: liuyy@csu.edu.cn.
\newline \indent \hspace*{-2.5mm}
$^{***}$ School of Mathematics and Statistics, Carleton
University, 1125 Colonel By Drive, Ottawa, ON Canada K1S 5B6,
E-mail: zhao@math.carleton.ca.
}\ \ and \ \ Yuanyuan Liu$^{**}$ and \ \ Yiqiang Q. Zhao$^{***}$}

\date{\today}

\begin{document}

\maketitle
\begin{abstract}
In this paper, we are devoted to  developing  matrix-analytic methods for solving Poisson's equation for irreducible and positive recurrent discrete-time Markov chains (DTMCs).
Two special solutions, including  the deviation matrix $D$ and the expected additive-type functional matrix $K$,  will be considered.
The results are applied to Markov chains of $GI/G/1$-type and $MAP/G/1$ queues with negative customers.
Further extensions to continuous-time Markov chains (CTMCs) are also investigated.
\vskip 0.2cm
\noindent \textbf{Keywords:}\ \ Markov chains, Poisson's equation, matrix-analytic methods,
Markov chains of $GI/G/1$-type, the deviation matrix, the expected additive-type functional matrix, $MAP/G/1$ queues
\vskip 0.2cm
\noindent {\bf AMS 2010 Subject Classification:} \ \ 60J10, 60J22, 60J27.
\end{abstract}

\section{Introduction}
Let $P = (P(i, j))_{i,j\in E}$ be the transition matrix of a DTMC $\mathbf{\Phi}=\{\Phi_k, k\geq 0\}$ on a countable state space $E$.
It is assumed that $P$ is irreducible and positive recurrent   with the unique invariant probability vector ${\boldsymbol\pi}$
such that  ${\boldsymbol\pi}^TP={\boldsymbol\pi}^T$ and ${\boldsymbol\pi}^T{\boldsymbol e}=1$, where $\boldsymbol e$ is a column vector of ones.
Let $\boldsymbol{g}: E\rightarrow \mathbb{R}$ be a function (or vector) satisfying $\boldsymbol{\pi}^T|\boldsymbol{g}|<\infty$.
For a given transition matrix $P$ and a function $\boldsymbol{g}$, Poisson's equation is written as
\begin{equation}\label{poi-fun-0}
  (I-P)\boldsymbol{f}=\overline{\boldsymbol{g}},
\end{equation}
where $I$ is  the identity matrix and $\overline{\boldsymbol{g}}=\boldsymbol{g}-(\boldsymbol{\pi}^{T}\boldsymbol{g})\boldsymbol{e}$.
In general, we refer to  the functions $\boldsymbol{g}$ and $\boldsymbol{f}$ as  the forcing function and the solution of Poisson's equation $(\ref{poi-fun-0})$,  respectively.

Poisson's equation has an important influence on the development of  Markov chain theory.
In \cite{GO02}, Glynn and Ormoneit  established a Hoeffding's inequality,  which provides an upper bound for the tail probability of the law of large numbers,
for strong ergodic DTMCs via the solution of Poisson's equation.
Further progress along this direction can be found in \cite{T09,LL21} and their references.
Poisson's equation may also be associated to central limit theorems.
In \cite{GM96},  Glynn and Meyn pointed out that the solution of Poisson's equation can be used to express the variance constant,
which is a key parameter in the central limit theorem.
Please refer to \cite{LWX14,GI22} and references therein for recent developments in this filed.
In Markov decision processes, Poisson's equation was known as the  dynamic programming
equation, see e.g.,  \cite{B07,SSO13}, and the functions $\boldsymbol{g}$ and $\boldsymbol{f}$ were called the cost function and the value function, respectively.
Poisson's equation is also applied to perturbation theory \cite{L12,JLT17}, augmented truncation approximations \cite{M17,LL18}, machine learning \cite{MV18,MV19} and others.

For real applications, it is crucial to solve  or estimate the solution of Poisson's equation.
In the literature, the solution of Poisson's equation had been  investigated  for birth-death processes \cite{L11,N21}, single-birth processes \cite{CZ04,JLY14},
single-death processes \cite{WZ20,LLZ21}, $M/G/1$  queues \cite{P94}, $PH/PH/1$ queues \cite{B96}, among others.
In addition, there are some approximate schemes for the solution of Poisson's equation, see e.g., \cite{MS02,MV18}.
Recently, Liu et al. \cite{LLZ21}  established augmented truncation approximations for the solution of Poisson's equation.

In this paper, we use matrix-analytic methods to solve Poisson's equation.
Since Neuts \cite{N81,N89} introduced and studied matrix-analytic methods for stochastic models,
matrix-analytic methods  had been widely used in queueing theory, supply chain management, inventory theory,  reliability, telecommunications networks,
risk and insurance analysis, finance mathematics, and biostatistics, see e.g., \cite{GV99,A03,H14}.
In \cite{DLL13}, Dendievel et al. used matrix-analytic methods to solve Poisson's equation for quasi-birth-and-death (QBD) processes.
Further progress has been made in \cite{LWX14, BDLM16,MKK2018}.
Here, we extend matrix-analytic methods in solving Poisson's equation for general Markov chains.
Specifically, we try to find a matrix $X$, with which we may represent the solution $\boldsymbol{f}$ to Poisson's equation (\ref{poi-fun-0}) in the form of $\boldsymbol{f}=X\boldsymbol{g}$.
In this sense, we rewrite Poisson's equation (\ref{poi-fun-0}) as the following matrix form
\begin{equation}\label{poi-fun-1}
  (I-P)X=I-\boldsymbol{e\pi}^T.
\end{equation}

Our  first  main result,  presented in Theorem \ref{the-result-general-MC}, gives a general matrix solution  $\widetilde{X}$ that satisfies Poisson's equation (\ref{poi-fun-1}).
Moreover, the matrix $\widetilde{X}$ is a unique solution in the sense of up to a constant matrix under some additional conditions,  i.e. $X=\widetilde{X}+\boldsymbol{e\beta}^T$,
where $\boldsymbol{\beta}$ is an arbitrary constant vector.
Then, two special solutions, which are called the deviation matrix $D$, see e.g., \cite{CD2002,DLL13},
and the expected additive-type functional matrix $K$, see e.g., \cite{P94,MKK2018},  will be investigated.
Particularly, we will focus on the latter because it needs a weaker existence condition than that for the deviation matrix $D$.

To apply our results, we consider the matrix solution $\widetilde{X}$ for Markov chains of $GI/G/1$-type,
which are class of block-structured Markov chians with many applications in queueing theory, see e.g., \cite{GH90, Z2000,MTZ12,H14}.
Mathematically, a DTMC $\mathbf{\Phi}$ on state space $E=\bigcup\nolimits_{i=0}^\infty {\ell (i)}$ is called a  Markov chain of $GI/G/1$-type if its transition probability matrix $P$  is given by
\begin{eqnarray}\label{P-of-GIG1}
&&P=\left[
\begin{array}{ccccc}
B_0 & B_1 & B_2 & B_3 & \cdots \\
B_{-1} & A_0 & A_1 & A_2 &\cdots \\
B_{-2} & A_{-1} & A_{0} & A_1 & \cdots \\
B_{-3} & A_{-2} & A_{-1} & A_0 & \cdots\\
\vdots & \vdots & \vdots & \vdots &\ddots
\end{array}
\right],
\end{eqnarray}
where $\ell(i):=\{(i,j), i\ge 0,1\le j \le m\}$ denotes the level set,
the matrix sequences $\{B_{i}, i\in \mathbb{Z}\}$ and $\{A_{i}, i\in \mathbb{Z}\}$ are non-negative matrices of size $m<\infty$,
where $\mathbb{Z}:=\{0,\pm1,\pm2,\cdots\}$.
Suppose that $\sum_{i=0}^{\infty} B_{i}$, $B_{-j} + \sum_{i=-j+1}^{\infty} A_{i}$ and $\sum_{i=-\infty}^{\infty} A_{i}$ are stochastic.
Markov chains of $GI/G/1$-type include QBD processes if $B_i=B_{-i}=A_i=A_{-i}= 0$ for $i\geq2$,
Markov chains of $GI/M/1$ type if $B_i = A_i = 0$ for $i\geq2$ and Markov chains of $M/G/1$ type if $B_{-i} = A_{-i} = 0$ for $i\geq2$.

The rest of this paper is organized in to 6 sections.
Section 2  introduces preliminaries of the solution of  Poisson's equation and the censored Markov chains, which play a key role for the subsequently proposed  matrix-analytic methods.
In section 3, we present a general matrix solution $\widetilde{X}$ and specific matrix solution $K$ of Poisson's equation  for  DTMCs.
Section 4 applies the results in section 3 to Markov chains of $GI/G/1$-type.
In section 5, we  give numerical calculations of $MAP/G/1$ queues with negative customers.
Section 6 considers the extension of the matrix solution $\widetilde{X}$  for  DTMCs to  CTMCs and section 7 presents concluding remarks.

\section{Preliminaries}
\label{sec:Prelim}
\subsection{The solution of Poisson's equation}
\label{subsec: Introduce Soultion}
For a  finite state Markov chain $\mathbf{\Phi}$, there exists a  unique solution of Poisson's equation (\ref{poi-fun-1}), see e.g., \cite{DLL13}.
Moreover, the solution is given by
\begin{equation}\label{finite-state}
X=(I-P)^{\#}+\boldsymbol{e\beta}^T,
\end{equation}
where $\boldsymbol\beta$ is an arbitrary constant vector that requires additional information to be determined
and $(I-P)^{\#}$ is  the group inverse, see e.g., \cite{Meyer1975,Barlow2000}.
The group inverse $W^{\#}$ of a finite square matrix $W$ is defined to be the unique matrix such that
\[
WW^{\#}W=W,\ \ W^{\#}WW^{\#}=W^{\#} \ \ and \ \ W^{\#}W=WW^{\#}.
\]
In the special case of $W=I-P$, the group inverse $(I-P)^{\#}$ can be easily determined by
\begin{equation}\label{compute-(I-P)}
(I-P)^{\#}=(I-P+\boldsymbol{e\pi}^T)^{-1}-\boldsymbol{e\pi}^T.
\end{equation}

For infinite state Markov chains,  the uniqueness of the solution does not necessarily hold. From Proposition 17.4.1 in \cite{MT09},
we obtain the following lemma, which presents a sufficient criterion for the uniqueness of the solution of Poisson's equation (\ref{poi-fun-1}).

\begin{lemma}\label{uniqueness-solution}
Let $\mathbf{\Phi}$  be an irreducible and positive recurrent Markov chain. Suppose that  $X_1$ and $X_2$ are two solutions of Poisson's equation (\ref{poi-fun-1})
with $\boldsymbol{\pi}^T(|X_1|+|X_2|)<\infty$. Then for some constant vector $\boldsymbol\beta$, we have $X_1=X_2+\boldsymbol{e\beta}^T$.
\end{lemma}

In general, the solution of Poisson's equation (\ref{poi-fun-1})  can be presented via the deviation matrix $D$ or the expected additive-type functional matrix $K$.
It is well known that the deviation matrix $D$ is defined as
\begin{equation}\label{definition-D}
D:=\sum_{n=0}^\infty \left(P^n - \boldsymbol{e \pi}^{T}\right),
\end{equation}
see e.g., \cite{CD2002}. It is not difficult to verify that $D$ satisfies Poisson's equation (\ref{poi-fun-1}) and $\boldsymbol{\pi}^{T} {D}=\boldsymbol{0}^T$,
where $\boldsymbol 0$ is the zero vector.
Moreover, we know that the elements of $D$ can be expressed in terms of the expected first return times:
\begin{equation}\label{element-D-1}
      D(i,i)=\pi(i)\left(\mathbb{E}_{\boldsymbol\pi}\left[\tau_i\right]-1\right)
     \end{equation}
and
 \begin{equation}\label{element-D-2}
D(j,i)=D(i,i)-\pi(i)\mathbb{E}_{j}\left[\tau_i\right],\ \ j\neq i,
     \end{equation}
where $\tau_i:=\inf\{k\geq 1:\Phi_k=i\}$  is the first return time to the state $i\in E$
and $\mathbb{E}_{\boldsymbol\pi}[\cdot]$ $($or $\mathbb{E}_{i}[\cdot]$$)$ denotes the  conditional expectation with respect to the initial distribution  $\boldsymbol\pi$ $($or initial state $i$$)$.

From (\ref{definition-D})--(\ref{element-D-2}) and \cite{DLL13},
we know that if the deviation matrix $D$ exists, then the chain must be aperiodic and $\mathbb{E}_{\boldsymbol\pi}\left[\tau_i\right]<\infty$ for some $i\in E$.
In fact, we can construct the expected additive-type functional matrix $K_\alpha=(K_\alpha(i,j))_{i,j\in E}$, which is  called the solution kernel in  Glynn \cite{P94},
that has a weaker existence condition than that for the deviation matrix $D$.
For a  fixed state $\alpha\in E$,
we define the matrix $K_\alpha$ such that
\begin{equation}\label{equ-add-functional}
K_\alpha(i,j):=\mathbb{E}_{i}\bigg[\sum_{k=0}^{\tau_{\alpha}-1}\overline{\mathbb{I}}(\Phi_k=j)\bigg],\ \ i,j\in E,
\end{equation}
where $\mathbb{I}(\cdot)$ denotes the indicator function and $\overline{\mathbb{I}}(\Phi_k=j)={\mathbb{I}}(\Phi_k=j)-\pi(j)$.
For the convenience of subsequent analysis, we simplify $K_\alpha$ to $K$.

\begin{lemma}\label{expected-additive-type-functional}
Let $\mathbf{\Phi}$ be an irreducible and positive recurrent Markov chain.
Then, the matrix  $K$ is one solution of Poisson's equation (\ref{poi-fun-1}) with $K(\alpha,j)=0$ for any state $j\in E$.
\end{lemma}
\begin{proof}  We first prove $K(\alpha,j)=0$, $j\in E$. It follows form Theorem 10.4.9 in \cite{MT09} that
\[\pi(\alpha)\mathbb{E}_{\alpha}\bigg[\sum_{k=0}^{\tau_{\alpha}-1}{\mathbb{I}}(\Phi_k=j)\bigg]=\pi(j).\]
From Kac's Theorem, we obtain
\[\pi(\alpha)=\frac{1}{\mathbb{E}_\alpha[\tau_\alpha]}.\]
Applying  (\ref{equ-add-functional}) yields $K(\alpha,j)=0$.

Now, we fix state $j\in E$.
Combining  (\ref{equ-add-functional}) and the strong Markov property, we have for any $i\in E$,
\[K(i,j)=\bar{\mathbb{I}}(i=j)+\sum_{ l\neq\alpha}P(i,l)K(l,j).\]
Since $K(\alpha,j)=0$, we obtain
\[K(i,j)=\bar{\mathbb{I}}(i=j)+\sum_{i\in E}P(i,l)K(l,j),\]
from which, we have
\[(I-P)K=I-\boldsymbol{e\pi}^T.\]

The proof is completed.
\end{proof}

\begin{remark}\label{f_d,f_k}
For a given function $\boldsymbol g$ satisfying $\boldsymbol\pi^T|\boldsymbol g|<\infty$, the functions $\boldsymbol{ f_D}:=D\boldsymbol g$ and $\boldsymbol{ f_K}:=K\boldsymbol g$ are solutions of  Poisson's equation (\ref{poi-fun-0}), which are given by
\begin{equation}\label{Dg}
{f_D}(i)=\mathbb{E}_{i}\bigg[\sum_{k=0}^{\infty}\overline{g}(\Phi_k)\bigg],\ \ i\in E,
\end{equation}
and
\begin{equation}\label{Kg}
{f_K}(i)=\mathbb{E}_{i}\bigg[\sum_{k=0}^{\tau_{\alpha}-1}\overline{g}(\Phi_k)\bigg],\ \ i\in E,
\end{equation}
respectively. See Remark 4 in \cite{GI22-2} and Theorem 3.3 in \cite{P94}.
Moreover, combining (\ref{Dg})--(\ref{Kg}) and the strong Markov property, we have ${f_D}(i)-{f_K}(i)={f_D}(\alpha)$ for any $i\in E$.
\end{remark}

\subsection{The censored Markov chain}
\label{subsec: censored chain}
In the following, we introduce the censoring technique.
Let $A$ be a non-empty subset of $E$.
Let $\theta_k$ be the $k^{th}$ time that $\mathbf{\Phi}$  successively visits a state in $A$,
i.e. $\theta_0:=\inf\{m\geq0: \Phi_m\in A\}$ and $\theta_{k+1}:=\inf\{m\geq \theta_{k}+1: \Phi_m\in A\}$.
The censored Markov chain $\mathbf{\Phi}^{(A)}=\{\Phi^{(A)}_k, k\geq0\}$ on  $A$
is defined by $\Phi^{(A)}_k=\Phi_{\theta_{k}}$, $k\geq0$, and its transition matrix and the invariant probability vector are denoted by $P^{(A)}$ and $\boldsymbol\pi^{(A)}$,  respectively.
Let $P_{_{A_1A_2}}=(P_{_{A_1A_2}}(i,j))_{i\in A_1,j\in A_2}$, where  $A_1$ and $A_2$ are subsets of $E$.
According to $A$ and its complement $B:=A^{C}$, we partition the transition matrix $P$ as
\begin{eqnarray}\label{partition-P}
    P = \bordermatrix[{[]}]{%
    & A & B \cr
A & P_{AA} & P_{AB} \cr
B & P_{BA} & P_{BB}
}.
\end{eqnarray}

From  section 5 in \cite{GV99}, we have the following lemma of the censored Markov chain.
\begin{lemma}\label{cen-p}
Let $\mathbf{\Phi}$ be an irreducible and positive recurrent Markov chain with the invariant probability vector ${\boldsymbol\pi}$ and let $A$ be a non-empty subset of $E$.
Then, the censored Markov chain $\mathbf{\Phi}^{(A)}$ is also irreducible and positive recurrent, whose transition probability matrix is given by
\begin{eqnarray}\label{P-E}
P^{(A)}=P_{AA}+P_{AB} \widehat{P}_{BB} P_{BA}
\end{eqnarray}
with $\widehat P_{BB} :=\sum_{k=0}^{\infty} P_{BB}^k$.  Moreover, the invariant probability vector $\boldsymbol\pi^{(A)}$ of  $\mathbf{\Phi}^{(A)}$ is given by
\begin{equation}\label{Pi-A}
\pi^{(A)}(i)=\frac{\pi(i)}{\sum_{j\in A}\pi(j)},\ \ i\in A.
\end{equation}
\end{lemma}

\begin{remark}\label{P_BB}
 (i) $\widehat{P}_{BB}$ is the minimal nonnegative solution of
  \[X(I-P_{BB})=(I-P_{BB})X=I.\]

(ii) $\widehat{P}_{BB}$ is finite since $P_{BB}$ is strictly substochastic.

(iii) $\widehat{P}_{BB}(i,j)$ is the expected number of visits to state $j\in B$ before entering $A$ given that the process starts from state $i\in B$, i.e.
  \begin{equation}\label{pro-P_BB}
\widehat{P}_{BB}(i,j)=\mathbb{E}_{i}\bigg[\sum_{k=0}^{\tau_{A}-1}{\mathbb{I}}(\Phi_k=j)\bigg],\ \ i,j\in  B,
\end{equation}
where $\tau_A:=\inf\{k\geq 1:\Phi_k\in A\}$  is the first return time to the set $A$.
\end{remark}

On the contrary,  we can also obtain the invariant probability vector $\boldsymbol\pi$ when we know  $\boldsymbol\pi^{(A)}$ and  $\widehat{P}_{BB}$.
Let us partition  $\boldsymbol{\pi}^T$ as $(\boldsymbol{\pi}_A^T, \boldsymbol{\pi}_B^T)$.
From (\ref{partition-P}) and $\boldsymbol{\pi}^{T}P=\boldsymbol{\pi}^{T}$, we have
\[\boldsymbol{\pi}_A^{T}P_{AB}+\boldsymbol{\pi}_B^{T}P_{BB}=\boldsymbol{\pi}_B^{T}.\]
Then,
\begin{equation}\label{pi-0}
\boldsymbol{\pi}_B^{T}(I-P_{BB})=\boldsymbol{\pi}_A^{T}P_{AB}.
\end{equation}
Postmultiplying both sides of (\ref{pi-0}) by $\widehat{P}_{BB}$, we obtain
\begin{equation}\label{pi-1}
\boldsymbol{\pi}_B^{T}=\boldsymbol{\pi}_A^{T}P_{AB}\widehat{P}_{BB}.
\end{equation}
According to $\boldsymbol\pi^T\boldsymbol e=1 $, we have
\begin{equation}\label{pi-2}
\boldsymbol{\pi}_A^T\left[I, P_{AB}\widehat{P}_{BB}\right]\boldsymbol e=1.
\end{equation}
Lemma \ref{cen-p} and  (\ref{pi-2}) yield
\begin{equation}\label{the constant-c}
\sum_{i\in A}\pi(i)=\left(\left(\boldsymbol\pi^{(A)}\right)^T\left[I, P_{AB}\widehat{P}_{BB}\right]\boldsymbol e\right)^{-1}.
\end{equation}
Thus, the  invariant probability vector $\boldsymbol\pi$ can be obtained by using (\ref{Pi-A}), (\ref{pi-1}) and  (\ref{the constant-c}).

\section{General Markov chains}
\label{sec:general}

In this section,  we will use matrix-analytic methods to solve  Poisson's equation (\ref{poi-fun-1}) for general Markov chains.
Let $O$  denote the zero matrix  with appropriate numbers of rows and columns.
Similar, the matrix $I$ and vector $\boldsymbol e$ defined previously will adapt to the dimensions in the following analysis.
Furthermore, let us partition the matrix $X$ as
\begin{equation}\label{partition-X}
X=\left[
\begin{array}{c}
X_A \\
X_B \\
\end{array}
\right].
\end{equation}

\begin{theorem}\label{the-result-general-MC}
Let $\mathbf{\Phi}$ be an irreducible and positive recurrent Markov chain and let $A$ be a finite non-empty subset of $E$.
Then, the matrix $\widetilde{X}$, given by
\begin{equation}\label{the-matrix-K-A}
\widetilde{X}_{A}=\left(I-{P}^{(A)}\right)^{\#}\left(\left[I, P_{AB}\widehat{P}_{BB}\right]-\left(\boldsymbol e+P_{AB}\widehat{P}_{BB}\boldsymbol e\right)\boldsymbol\pi^T\right)
\end{equation}
and
\begin{equation}\label{the-matrix-K-B}
\widetilde{X}_{B}=\widehat{P}_{BB}P_{BA}\widetilde{X}_A+\left[O,\widehat{P}_{BB}\right]-\widehat{P}_{BB}\boldsymbol e \boldsymbol \pi^T,
\end{equation}
is one solution of  Poisson's equation (\ref{poi-fun-1}).
Moreover, if $\mathbb{E}_{\boldsymbol\pi}\left[\tau_i\right]<\infty$ for some $i\in E$, then the matrix  $\widetilde{X}$
is the unique matrix solution of Poisson's equation (\ref{poi-fun-1}) in the set of matrices $X$ such that $\boldsymbol{\pi}^T|X|<\infty$.
\end{theorem}
\begin{proof}We first prove that $\widetilde{X}$ satisfies Poisson's equation (\ref{poi-fun-1}).
From (\ref{partition-P}) and (\ref{partition-X}), we rewrite Poisson's equation (\ref{poi-fun-1}) as
\begin{eqnarray*}
\left[
\begin{array}{cc}
I-P_{AA} & -P_{AB} \\
 -P_{BA}   & I-P_{BB} \\
\end{array}
\right]
\left[
\begin{array}{c}
X_A \\
X_B \\
\end{array}
\right] =
I-\boldsymbol{e\pi}^{T}.
\end{eqnarray*}
We obtain
\begin{equation}\label{X_0}
  (I-P_{AA}) X_A =P_{AB} X_B + \left[I,O\right]-{\boldsymbol{e\pi}^{T}}
\end{equation}
and
\begin{equation}\label{X_1}
  (I-P_{BB}) X_B =P_{BA} X_A + \left[O,{I}\right]-\boldsymbol{e\pi}^{T}.
\end{equation}
Premultiplying both sides of (\ref{X_1}) by $\widehat{P}_{BB}$ gives us
\begin{eqnarray}\label{X_B_0}
X_B=\widehat{P}_{BB}P_{BA}  X_A +  \widehat{P}_{BB}\left[O,I\right]-\widehat{P}_{BB}{\boldsymbol{e\pi}^{T}}.
\end{eqnarray}
Substituting (\ref{X_B_0}) into (\ref{X_0}), we have
\begin{eqnarray*}
 (I-P_{AA}) X_A &=& P_{AB}\widehat{P}_{BB}P_{BA}  X_A +  \left[O,P_{AB}\widehat{P}_{BB}\right]-P_{AB}\widehat{P}_{BB}{\boldsymbol{e\pi}^{T}}+ \left[I,O\right]-{\boldsymbol{e\pi}^{T}}\\
   &=&   P_{AB}\widehat{P}_{BB}P_{BA}  X_A+ \left[I,P_{AB}\widehat{P}_{BB}\right]-\left(\boldsymbol{e}+P_{AB}\widehat{P}_{BB}\boldsymbol{e}\right){\boldsymbol{\pi}^{T}}.
\end{eqnarray*}
Thus,  we have
\begin{equation}\label{X_A_0}
\left(I-P^{(A)}\right) X_{A}= \left[I,P_{AB}\widehat{P}_{BB}\right]-\left(\boldsymbol{e}+P_{AB}\widehat{P}_{BB}\boldsymbol{e}\right){\boldsymbol{\pi}^{T}}.
\end{equation}
It follows from (\ref{finite-state}) and (\ref{X_A_0}) that
\begin{equation}\label{X_A}
  X_{A}=\left(I-P^{(A)}\right)^{\#}\left( \left[I,P_{AB}\widehat{P}_{BB}\right]-\left(\boldsymbol{e}+P_{AB}\widehat{P}_{BB}\boldsymbol{e}\right){\boldsymbol{\pi}^{T}}\right)+{\boldsymbol{e\beta}^{T}}=\widetilde{X}_A+{\boldsymbol{e\beta}^{T}}.
\end{equation}
From (\ref{X_B_0}) and (\ref{X_A}), we obtain
\[
X_B=\widehat{P}_{BB}P_{BA} \left( \widetilde{X}_A+{\boldsymbol{e\beta}^{T}}\right) +  \widehat{P}_{BB}\left[O,I\right]-\widehat{P}_{BB}{\boldsymbol{e\pi}^{T}}.
\]
According to Remark \ref{P_BB} (iii), we know that $\widehat{P}_{BB}P_{BA}$ denotes the probability of first hitting $A$ from $B$.  We thus have
\[\widehat{P}_{BB}P_{BA}{\boldsymbol{e}}={\boldsymbol{e}}.\]
Then, we obtain
\begin{eqnarray}\label{X_B}
X_{B} =\widehat{P}_{BB}P_{BA}\widetilde{X}_A +  \left[O,\widehat{P}_{BB}\right]-\widehat{P}_{BB}{\boldsymbol{e\pi}^{T}}+{\boldsymbol{e\beta}^{T}}= \widetilde{X}_{B}+ \boldsymbol{e\beta}^{T}.
\end{eqnarray}

Combining (\ref{X_A}) with (\ref{X_B}), we find that ${X}=\widetilde{X}+\boldsymbol{e\beta}^{T}$, i.e. the matrix $\widetilde{X}$ is one solution of Poisson's equation (\ref{poi-fun-1}).

Now, we prove that under the condition of $\mathbb{E}_{\boldsymbol\pi}[\tau_i]<\infty$,  $\widetilde{X}$ defined by (\ref{the-matrix-K-A}--\ref{the-matrix-K-B}) satisfies $\boldsymbol{\pi}^T|\widetilde{X}|<\infty$, i.e.
\[\sum_{i\in E}\pi(i)\left|\widetilde{X}(i,n)\right|<\infty, \ \ n\in E.\]
Since $A$ is a finite set, we can find large enough positive constant $\gamma_n$ for any $n\in E$, such that
\[\sum_{i\in A}\left|\widetilde{X}(i,n)\right|<\gamma_n<\infty.\]
Since $\mathbf{\Phi}$ is irreducible and $\mathbb{E}_{\boldsymbol\pi}[\tau_i]<\infty$,
this shows that $\mathbb{E}_{\boldsymbol\pi}[\tau_A]<\infty$. Thus, we have
\begin{eqnarray*}
  \sum_{i\in E}\pi(i)\left|\widetilde{X}(i,n)\right| &\leq& \gamma_n\boldsymbol{\pi}_A^T\boldsymbol{e}+ \sum_{i\in B}\pi(i)\left|\widetilde{X}(i,n)\right|\\
   &\leq&  \gamma_n\boldsymbol{\pi}_A^T\boldsymbol{e}+\sum_{i\in B}\pi(i)\sum_{j\in A}\mathbb{P}_{i}(\tau_A=j)\left|\widetilde{X}(j,n)\right|\\
   &&+\sum_{i\in B}\pi(i)\mathbb{E}_{i}\left[\tau_{A}\right] +\sum_{i\in B}\pi(i)\mathbb{E}_{i}\bigg[\sum_{k=0}^{\tau_{A}-1}\mathbb{I}(\Phi_k\in B)\bigg]\\
   &\leq&\gamma_n\boldsymbol{\pi}_A^T\boldsymbol{e}+\gamma_n\boldsymbol{\pi}_B^T\boldsymbol{e}+2\mathbb{E}_{\boldsymbol\pi}[\tau_A]\\
   &=&\gamma_n+2\mathbb{E}_{\boldsymbol\pi}[\tau_A]<\infty,
\end{eqnarray*}
where $\mathbb{P}_{i}(\cdot)$ denotes the  conditional probability with respect to the  initial state $i$.

By Lemma \ref{uniqueness-solution}, all matrix solutions  $X$ of $(I-P)X = I-\boldsymbol{e\pi}^T$
 such that $\boldsymbol{\pi}^T|X|<\infty$ are given by $X = \widetilde{X} +\boldsymbol{e\beta}^{T}$ for some constant vector $\boldsymbol{\beta}$. This completes the proof of Theorem \ref{the-result-general-MC}.
\end{proof}

\begin{remark}
For an irreducible and  positive recurrent Markov chain $\mathbf{\Phi}$,
$\mathbb{E}_{\boldsymbol\pi}[\tau_i]<\infty$ for some (then for every)  $i \in E$ is equivalent to $\mathbb{E}_{i}[\tau_i^2]<\infty$, see Lemma 2.6 in \cite{{DLL13}}.
\end{remark}

Under the uniqueness assumption, we can use the matrix solution $\widetilde{X}$ to characterize the deviation matrix $D$.

\begin{corollary}\label{expression-D}
Let $\mathbf{\Phi}$ be an irreducible, positive recurrent and aperiodic  Markov chain. If $\mathbb{E}_{\boldsymbol\pi}\left[\tau_i\right]<\infty$ for some $i\in E$, then
 $D=(I-\boldsymbol{e\pi}^{T})\widetilde{X}$,
where $\widetilde{X}$ is determined by (\ref{the-matrix-K-A})--(\ref{the-matrix-K-B}).
\end{corollary}
\begin{proof}
Since $\mathbb{E}_{\boldsymbol\pi}\left[\tau_i\right]<\infty$, we know that $D$ exists and $D$ is one solution of Poisson's equation  (\ref{poi-fun-1}).
Moreover, it follows from (\ref{element-D-1})--(\ref{element-D-2}) that
$$\sum_{i\in E}\pi(i)|D(i,n)|\leq |D(n,n)|+\pi(n)\mathbb{E}_{\boldsymbol\pi}\left[\tau_n\right]<\infty$$
for any $n\in E$, i.e. $\boldsymbol\pi|D|<\infty$. From Theorem \ref{the-result-general-MC}, we have
\[D=\widetilde{X}+\boldsymbol{e\beta}^T.\]
Since  $\boldsymbol{\pi}^TD=\boldsymbol{0}$, we obtain $\boldsymbol{\beta}=-\boldsymbol{\pi}^T\widetilde{X}$. Thus, this proof is completed.
\end{proof}

In fact, we could have derived the matrix $K$ under the same condition $\mathbb{E}_{\boldsymbol\pi}\left[\tau_i\right]<\infty$ by using the same averments of Corollary \ref{expression-D}.
However, we can derive the matrix $K$ without the condition $\mathbb{E}_{\boldsymbol\pi}\left[\tau_i\right]<\infty$ by using different arguments.

For a finite non-empty subset $A$ of $E$, let $N=(N(i,j))_{i,j\in E}$ denote a matrix such that
\begin{equation*}
N(i,j):=\mathbb{E}_{i}\bigg[\sum_{k=0}^{\tau_{A}-1}\overline{\mathbb{I}}(\Phi_k=j)\bigg], \ \ i,j\in E.
\end{equation*}
The following lemma reveals a relationship between the matrix $K$ and the matrix $N$.
\begin{lemma}\label{additive-function-set}
Let $\mathbf{\Phi}$ be an irreducible and positive recurrent Markov chain and let $A$ be a finite non-empty subset of $E$.
Then, for any fixed  state $\alpha\in A$,  we have
\begin{equation}\label{expression-add-functional-1}
K(i,j) =N(i,j)+\sum_{l\in A}\mathbb{P}_{i}(\Phi_{\tau_A}=l)K(l,j),\ \ i,j\in E.
\end{equation}
\end{lemma}
\begin{proof}
Since $\alpha\in A$, $\tau_\alpha=\max\{\tau_\alpha,\tau_A\}$.  Using (\ref{equ-add-functional}) and the strong Markov property,
we obtain
\begin{eqnarray*}
   K(i,j) &=&\mathbb{E}_{i}\bigg[\sum_{k=0}^{\max\{\tau_\alpha,\tau_A\}-1}{\overline{\mathbb{I}}}(\Phi_k=j)\bigg] \\
   &=& \mathbb{E}_{i}\bigg[\sum_{k=0}^{\tau_A-1}{\overline{\mathbb{I}}}(\Phi_k=j)\bigg]
   +\left(\mathbb{E}_{i}\bigg[\sum_{k=\tau_A}^{\tau_\alpha-1}{\overline{\mathbb{I}}}(\Phi_k=j)\bigg]\right)\mathbb{I}(\tau_\alpha>\tau_A)\\
   &=&N(i,j)+\sum_{l\in A,l\neq\alpha}\mathbb{P}_{i}(\Phi_{\tau_A}=l)\left(\mathbb{E}_{l}\bigg[\sum_{k=0}^{\tau_\alpha-1}{\overline{\mathbb{I}}}(\Phi_k=j)\bigg]\right)\\
   &=&N(i,j)+\sum_{l\in A,l\neq\alpha}\mathbb{P}_{i}(\Phi_{\tau_A}=l)K(l,j).
\end{eqnarray*}
From Lemma \ref{expected-additive-type-functional}, we have
\[K(\alpha, j)=0.\]
Thus,
\[K(i,j)=N(i,j)+\sum_{l\in A}\mathbb{P}_{i}(\Phi_{\tau_A}=l)K(l,j).\]

The proof is completed.
\end{proof}

\begin{theorem}\label{result-expression-add-functional}
Let $\mathbf{\Phi}$ be an irreducible and positive recurrent Markov chain and let $A$ be a finite non-empty subset of $E$.
Then, for any fixed state $\alpha\in A$,
we have $K=\widetilde{X}-\boldsymbol{e}\boldsymbol{\beta}^T$,
where $\beta(j)=\widetilde{X}(\alpha,j)$ for any $j\in E$ and $\widetilde{X}$ is determined by (\ref{the-matrix-K-A})--(\ref{the-matrix-K-B}).
\end{theorem}
\begin{proof}
In (\ref{expression-add-functional-1}), we consider the two cases of $i\in A$ and $i\in B=A^C$, separately.
For the former case,
let  $M_A=(M_A(i,j))_{i,j\in A}$ denote the matrix such that
\[M_A(i,j):=\mathbb{P}_{i}(\Phi_{\tau_A}=j).\]
It is easy to verify  that
\begin{eqnarray}\label{expression-add-functional-2}
 \nonumber M_A &=& P_{AA}+P_{AB}\sum_{n=0}^{\infty}{P}^n_{BB}P_{BA}\\
  \nonumber   &=&  P_{AA}+P_{AB}\widehat{P}_{BB}P_{BA}\\
      &=& P^{(A)}.
\end{eqnarray}
Combining (\ref{expression-add-functional-1}) and (\ref{expression-add-functional-2}), we have
\begin{equation}\label{F_A}
(I-P^{(A)})K_A = N_A.
\end{equation}
It follows from (\ref{finite-state}) and (\ref{F_A}) that
\begin{equation*}\label{expression-add-functional-3}
K_A=(I-{P}^{(A)})^{\#}N_A+\boldsymbol e\boldsymbol\beta^T.
\end{equation*}
Our task now is to solve the matrix $N_A$. For any state $j\in A$, we have
\begin{equation}\label{N_AA}
\mathbb{E}_{i}\bigg[\sum_{k=0}^{\tau_{A}-1}\mathbb{I}(\Phi_k=j)\bigg]=\mathbb{I}(i=j).
\end{equation}
If $j\in B$, it follows from the strong Markov property and Remark \ref{P_BB} (iii)  that
\begin{equation}\label{N_AB}
\mathbb{E}_{i}\bigg[\sum_{k=0}^{\tau_{A}-1}\mathbb{I}(\Phi_k=j)\bigg]=\sum_{l\in B}P_{AB}(i,l)\widehat{P}_{BB}(l,j).
\end{equation}
From the strong Markov property, we have for $i\in A$,
\begin{equation}\label{E-A-tau-A}
 \mathbb{E}_i\left[\tau_{A}\right] = {1}+\sum_{l\in B}P_{AB}(i,l)\mathbb{E}_{l}\left[\tau_A\right].
\end{equation}
According to Remark \ref{P_BB} (iii), it is clear that for $i\in B$,
\begin{equation}\label{E-B-tau-A}
 \mathbb{E}_i\left[\tau_{A}\right]=\sum_{j\in B}\mathbb{E}_{i}\bigg[\sum_{k=0}^{\tau_{A}-1}\mathbb{I}(\Phi_k=j)\bigg]=\sum_{j\in B}\widehat{P}_{BB}(i,j).
\end{equation}
Combining (\ref{N_AA})--(\ref{E-B-tau-A}), we obtain
\begin{equation}\label{N_A}
N_A=\left[I, P_{AB}\widehat{P}_{BB}\right]-\left(\boldsymbol e+P_{AB}\widehat{P}_{BB}\boldsymbol e\right)\boldsymbol\pi^T,
\end{equation}
from which, we have
\begin{equation}\label{K_a}
 K_A=\left(I-{P}^{(A)}\right)^{\#}\left(\left[I, P_{AB}\widehat{P}_{BB}\right]-\left(\boldsymbol e+P_{AB}\widehat{P}_{BB}\boldsymbol e\right)\boldsymbol\pi^T\right)+\boldsymbol e\boldsymbol\beta^T=\widetilde{X}_A+\boldsymbol e\boldsymbol \beta^T .
\end{equation}

Now, we consider the case of $i\in B$.
Let $M_{B}=(M_{B}(i,j))_{i\in B,j\in A}$ denote  the matrix such that
\[M_{B}(i,j):=\mathbb{P}_{i}(X_{\tau_A}=j),\]
and
\begin{equation}\label{M_B}
M_{B} =\sum_{n=0}^{\infty}{P}^n_{BB}P_{BA} =\widehat{P}_{BB}P_{BA}.
\end{equation}
From (\ref{expression-add-functional-1}) and (\ref{K_a}), we have
\begin{eqnarray*}
 K_B &=& N_B+\widehat{P}_{BB}P_{BA}\left(\widetilde{K}_A+{\boldsymbol e}{\boldsymbol \beta}^T\right) \\
   &=& N_B+\widehat{P}_{BB}P_{BA}\widetilde{K}_A+\widehat{P}_{BB}P_{BA}{\boldsymbol e}{\boldsymbol\beta}^T\\
   &=&N_B+\widehat{P}_{BB}P_{BA}\widetilde{K}_A+{\boldsymbol e}{\boldsymbol \beta}^T.
\end{eqnarray*}
For any state $j\in A$, it is easy to see that
\begin{equation}\label{N_BA}
\mathbb{E}_{i}\bigg[\sum_{k=0}^{\tau_{A}-1}\mathbb{I}(\Phi_k=j)\bigg]=0.
\end{equation}
According to (\ref{N_BA}), Remark \ref{P_BB} (iii) and (\ref{E-B-tau-A}), it follows that
\begin{equation}\label{N_B}
N_B=\left[O,\widehat{P}_{BB}\right]-\widehat{P}_{BB}\boldsymbol e \boldsymbol \pi^T,
\end{equation}
from which, we have
\begin{equation*}
K_B = \left[O,\widehat{P}_{BB}\right]-\widehat{P}_{BB}\boldsymbol e \boldsymbol \pi^T+\widehat{P}_{BB}P_{BA}\widetilde{X}_A+{\boldsymbol e}{\boldsymbol\beta}^T
   = \widetilde{X}_B+{\boldsymbol e}{\boldsymbol \beta}^T.
\end{equation*}

Finally, we obtain $\beta(j)=-\widetilde{X}(\alpha,j)$ for any $j\in E$  by Lemma \ref{expected-additive-type-functional}.
This completes the proof of Theorem \ref{result-expression-add-functional}.
\end{proof}

\begin{remark}
(i) Combining Theorem \ref{the-result-general-MC} and  (\ref{N_A}) and (\ref{N_B}), we have
\begin{equation}\label{probabilistic-1}
\widetilde{X}_A=\left(I-P^{(A)}\right)^{\#}N_A,
\end{equation}
\begin{equation}\label{probabilistic-2}
\widetilde{X}_B=\widehat{P}_{BB}P_{BA}\left(I-P^{(A)}\right)^{\#}N_A+N_B.
\end{equation}
From (\ref{probabilistic-1})--(\ref{probabilistic-2}), we find that the matrix $\widetilde{X}$  consists of the group inverse  $\left(I-P^{(A)}\right)^{\#}$ and the matrix $N$.
According to (\ref{finite-state}), we know that $\left(I-P^{(A)}\right)^{\#}$ is the solution of Poisson's equation for the censored Markov chain $\mathbf{\Phi}^{(A)}$.
Thus, we connect the solution $\widetilde{X}$ of Poisson's equation for $\mathbf\Phi$
and the solution $\left(I-P^{(A)}\right)^{\#}$ of Poisson's equation for $\mathbf{\Phi}^{(A)}$ through the matrix $N$.

(ii) In addition, if the set $A$ is an atom, i.e. $P(i,k)=P(j,k)$ for any $i,j\in A$ and $k\in E$,
the matrix $N=K$ is a solution of Poisson's equation (\ref{poi-fun-1}) and satisfies $N_A=O$.
For this case, we have $\widetilde{X}=K=N$.
\end{remark}

\section{Markov chains  of GI/G/1-type}
\label{sec:GIG1}
In this section, we apply our results to Markov chains  of $GI/G/1$-type with the transition matrix given by (\ref{P-of-GIG1}).
For simplicity, we write
\[L_{\leq i}=\bigcup_{k=0}^{i} \ell (k),\]
and $L_{\geq i}$ for the complement of $L_{\leq (i-1)}$. For the block-structured matrix $M$, we write
$M_{\ell_i\ell_j}=(M_{\ell_i\ell_j}(k,l))_{k\in\ell_i,l\in\ell_j}$ as $M_{ij}$ for convenience.

We now introduce $R$-measures $R_{i,j}$ and $G$-measures $G_{i,j}$, which are helpful to studying  Markov chains of $GI/G/1$-type, see e.g., \cite{Z2000,zhao2003}.
For $0\leq i < j$, $R_{i,j}$ is defined as a matrix of size $m\times m$
whose $(k, l)^{th}$ entry is the expected number of visits to state $(j, l)$ before
hitting any state in $L_{\leq (j-1)}$, given that the process starts in state $(i, k)$, i.e
\[R_{i,j}(k,l)  :=  \mathbb{E}_{(i,k)}\bigg[\sum_{t=0}^{\tau_{L_{\leq j-1}}} \mathbb{I}(\Phi_t=(j,l))\bigg].\]
For $i > j\geq 0$, $G_{i,j}$ is defined as a matrix of size $m\times m$ whose $(k, l)^{th}$ entry
is the probability of hitting state $(j, l)$ when the process enters $L_{\leq (i-1)}$ for
the first time, given that the process starts in state $(i, k)$, i.e.
\[G_{i,j}(k,l) := \mathbb{P}_{(i,k)}\left[\tau_{L_{\leq (i-1)}}< \infty, \Phi_{\tau_{L_{\leq (i-1)}}}=(j,l)\right].\]
Due to the property of repeating rows, we can write simply $R_{i,n}=R_{n-i}$ and $G_{n,i}=G_{n-i}$ for $i\geq 1$.

Let $P^{(n)}=P^{(L_{\leq n})}$ and $\boldsymbol\pi^{(n)}$ be the transition matrix and the invariant probability vector of the censored Markov chain with censoring set $L_{\leq n}$ for $n\geq 0$, respectively.
Then, we know from \cite{Z2000} that
\begin{equation*}
  P^{(n)}_{i,j}= P^{(n+1)}_{i+1,j+1}=\cdots, \ \ \textrm{for all} \ i,j=1,2,\cdots,n.
\end{equation*}
Thus, for any $i\geq 0$, we can define
\begin{equation}\label{Phi-ij}
  \Psi_i = P^{(n)}_{n-i,n} \ \ \mbox{and} \ \  \Psi_{-i} = P^{(n)}_{n,n-i}, \ \ \textrm{for}\ n>i.
\end{equation}
Furthermore, we have
\[
 R_i=\Psi_i(I-\Psi_0)^{-1},\ \ G_i=(I-\Psi_0)^{-1} \Psi_{-i},\ \ i\geq 1.
\]

In fact, the matrices $R_i$, $G_i$, $R_{0,i}$ and $G_{i,0}$ can be used to represent the matrix $\Psi_0$. From Theorem 10 and Theorem 12 in \cite{Z2000} , we have the following lemma.
\begin{lemma}\label{R-G-PHI}
Let $\mathbf{\Phi}$ be an irreducible and positive recurrent
 Markov chain of  $GI/G/1$-type. Then, we have
\[R_i(I-\Psi_0)=A_i+\sum_{k=1}^{\infty}R_{i+k}(I-\Psi_0)G_k, \ \ i\geq1,\]
\[(I-\Psi_0)G_i=A_{-i}+\sum_{k=1}^{\infty}R_{k}(I-\Psi_0)G_{i+k}, \ \ i\geq1,\]
\[\Psi_0=A_0+\sum_{k=1}^{\infty}R_k(I-\Psi_0)G_k,\]
and
\[R_{0,i}(I-\Psi_0)=B_i+\sum_{k=1}^{\infty}R_{0,i+k}(I-\Psi_0)G_k, \ \ i\geq1,\]
\[(I-\Psi_0)G_{i,0}=B_{-i}+\sum_{k=1}^{\infty}R_{k}(I-\Psi_0)G_{i+k,0}, \ \ j\geq1,\]
\[\Psi_0=B_0+\sum_{k=1}^{\infty}R_{0,k}(I-\Psi_0)G_{k,0}.\]
\end{lemma}

If the Markov chain of $GI/G/1$-type  $\mathbf{\Phi}$ is irreducible and positive recurrent,
then the invariant probability vector $\boldsymbol\pi$  can be expressed in terms of $R$-measures, see \cite{GH90}:
\begin{equation}\label{rec-pi}
  \boldsymbol{\pi}_n^{T}=\boldsymbol{\pi}_0^{T}R_{0,n}+\sum_{k=1}^{n-1}\boldsymbol{\pi}_k^{T}R_{n-k}, \ \ n\geq1,
\end{equation}
where we denote ${\boldsymbol\pi}_{\ell_n}$ by ${\boldsymbol\pi}_{n}$ for simplicity.

The matrix $H$, which is obtained by deleting the first block row and the first block column of $P$ for Markov chains of GI/G/1-type,  is given by
\begin{eqnarray}\label{H}
H=\left[
\begin{array}{ccccc}
 A_0 & A_1 & A_2 & A_3 &\cdots \\
 A_{-1} & A_0 & A_1& A_2  & \cdots \\
 A_{-2} & A_{-1} & A_0& A_1  & \cdots\\
 A_{-3} & A_{-2} & A_{-1}& A_0  & \cdots\\
  \vdots & \vdots & \vdots & \vdots &\ddots
\end{array}
\right].
\end{eqnarray}

From Theorem 9 in \cite{zhao2003}, we have the following Lemma.
\begin{lemma}\label{fun-H}
For the matrix $H$ defined by (\ref{H}),
the matrix $\widehat{H}$ is recursively given by
\begin{eqnarray*}
\widehat{H}_{ij} {\rm{ = }}\left\{ {\begin{array}{cc}
{\sum \limits_{n=1}^{i-1} G_{i-n} \widehat{H}_{nj},} & i>j,\\
{\widehat{H}_{11}+\sum \limits_{n=1}^{i-1} G_{i-n} \widehat{H}_{nj}=\widehat{H}_{11}+\sum \limits_{n=1}^{j-1}\widehat{H}_{in} R_{j-n}, } & i=j, \\
{\sum \limits_{n=1}^{j-1}\widehat{H}_{in} R_{j-n}, } & i<j,
\end{array}} \right.
\end{eqnarray*}
where $\widehat{H}_{11}=(I-\Psi_0)^{-1}$.
\end{lemma}

\begin{theorem}\label{X-GIG1}
Let $\mathbf{\Phi}$ be an irreducible and positive recurrent Markov chain of $GI/G/1$-type and
let $A=\ell(0)$. Then, we have
\begin{equation*}
 \widetilde{X}_{ij}=\left\{\aligned & \left(I-P^{(0)}\right)^{\#}\left(I-\left(I+\sum_{m=1}^{\infty}\sum_{n=1}^{\infty}B_{n}\widehat{H}_{nm}\right) \boldsymbol{e\pi}_0^{T}\right), & i=0,j=0,\\
  &\widetilde{X}_{00} \sum_{n=1}^{\infty}B_n\widehat{H}_{nj}, & i=0,j>1,\\
  &\left(\sum \limits_{m=1}^{\infty} {\widehat{H}_{im}B_{-m}}\right)\widetilde{X}_{00}
-\left(\sum \limits_{m=1}^{\infty} \widehat{H}_{im}\right)\boldsymbol{e\pi}_0^{T}, & i>1,j=0,\\
&\left(\sum \limits_{m=1}^{\infty} {\widehat{H}_{im}B_{-m}}\right)\widetilde{X}_{0j}+\widehat{H}_{ij}-
\left(\sum \limits_{m=1}^{\infty} \widehat{H}_{im}\right)\boldsymbol{e\pi}_j^{T}, & i>1,j>1.
 \endaligned \right.
\end{equation*}
\end{theorem}
\begin{proof} Since $A=\ell(0)$, it is easy to see that
\begin{eqnarray}
\left.
\begin{array}{cc}
P_{AA}=B_0, & P_{AB}=[B_1,B_2,B_3,\cdots], \\
\\
 P_{BA}=\left[
\begin{array}{cccc}
 B_{-1}\\
 B_{-2} \\
 B_{-3}\\
  \vdots
\end{array}
\right],   & P_{BB}=H=\left[
\begin{array}{cccc}
 A_0 & A_1 & A_2 &\cdots \\
 A_{-1} & A_0 & A_1 & \cdots \\
 A_{-2} & A_{-1} & A_0 & \cdots\\
  \vdots & \vdots & \vdots &\ddots
\end{array}
\right]. \\
\end{array}
\right.
\end{eqnarray}
Moreover, according to (\ref{partition-X}), we write
\[\widetilde{X}_A=[\widetilde{X}_{00},\widetilde{X}_{01},\widetilde{X}_{02},\cdots], \ \
\widetilde{X}_B=\left[
\begin{array}{cccc}
\widetilde{X}_{10} & \widetilde{X}_{11} & \widetilde{X}_{12} & \cdots \\
\widetilde{X}_{20} & \widetilde{X}_{21} & \widetilde{X}_{22} & \cdots \\
\widetilde{X}_{30} & \widetilde{X}_{31} & \widetilde{X}_{32} & \cdots \\
\vdots & \vdots & \vdots & \ddots
\end{array}
\right].
\]
From Lemma \ref{fun-H}, it is easy to obtain
\[
P_{AB}\widehat{P}_{BB}=\left[\sum_{n=1}^{\infty}B_{n}\widehat{H}_{n1}, \sum_{n=1}^{\infty}B_{n}\widehat{H}_{n2}, \sum_{n=1}^{\infty}B_{n}\widehat{H}_{n3}, \cdots\right].
\]
According to Theorem \ref{the-result-general-MC}, it follows that
\[\widetilde{X}_A=\left(I-P^{(0)}\right)^{\#}\left(\left[I,\sum_{n=1}^{\infty}B_{n}\widehat{H}_{n1},\sum_{n=1}^{\infty}B_{n}\widehat{H}_{n2},\cdots\right]
-\left(I+\sum_{m=1}^{\infty}\sum_{n=1}^{\infty}B_{n}\widehat{H}_{nm}\right) \boldsymbol{e\pi}^{T}\right).\]
Furthermore,  we obtain
\begin{equation}\label{GIG1_D_00}
\widetilde{X}_{00}=\left(I-P^{(0)}\right)^{\#}\left(I-\left(I+\sum_{m=1}^{\infty}\sum_{n=1}^{\infty}B_{n}\widehat{H}_{nm}\right) \boldsymbol{e\pi}_0^{T}\right)
\end{equation}
and
\begin{eqnarray}\label{GIG1_D_0k}
\widetilde{X}_{0j}&=&\left(I-P^{(0)}\right)^{\#}\left(\sum_{n=1}^{\infty}B_{n}\widehat{H}_{nj}-\left(I+\sum_{m=1}^{\infty}\sum_{n=1}^{\infty}B_{n}\widehat{H}_{nm}\right) \boldsymbol{e\pi}_j^{T}\right).
\end{eqnarray}
From (\ref{pi-1}) and Lemma \ref{fun-H}, we obtain
\begin{equation}\label{pi-j}
  \boldsymbol{\pi}_j^{T}=\boldsymbol{\pi}_0^{T}\sum_{n=1}^{\infty}B_n\widehat{H}_{nj}.
\end{equation}
Substituting (\ref{pi-j}) into (\ref{GIG1_D_0k}), we have, for $j\geq 1$,
\begin{eqnarray}\label{D_0k1}
\widetilde{X}_{0j}&=&\left(I-P^{(0)}\right)^{\#}\left(\sum_{n=1}^{\infty}B_{n}\widehat{H}_{nj}
-\left(I+\sum_{m=1}^{\infty}\sum_{n=1}^{\infty}B_{n}\widehat{H}_{nm}\right)
\boldsymbol{e\pi}_0^{T}\sum_{n=1}B_n\widehat{H}_{nj}\right)\nonumber\\
&=&\left(I-P^{(0)}\right)^{\#}\left(I
-\left(I+\sum_{m=1}^{\infty}\sum_{n=1}^{\infty}B_{n}\widehat{H}_{nm}\right)
\boldsymbol{e\pi}_0^{T}\right)\sum_{n=1}B_n\widehat{H}_{nj}\nonumber\\
&=&\widetilde{X}_{00}\sum_{n=1}^{\infty}B_n\widehat{H}_{nj}\nonumber.
\end{eqnarray}

From Theorem \ref{the-result-general-MC} again, we have
\[
\widetilde{X}_B=\widehat{H}P_{AB}\widetilde{X}_A + [O,\widehat{H}]-\widehat{H}\boldsymbol{e\pi}^{T}.
\]
Thus, for $i\geq 1$,
\[\widetilde{X}_{i0} = \left(\sum_{m=1}^{\infty}{\widehat{H}_{im}B_{-m}}\right)\widetilde{X}_{00} - \left(\sum \limits_{m=1}^{\infty} \widehat{H}_{im}\right)\boldsymbol{e\pi}_0^{T} \]
Similarly, we have
\begin{eqnarray*}
\widetilde{X}_{ij}=\left(\sum \limits_{m=1}^{\infty} {\widehat{H}_{im}B_{-m}}\right)\widetilde{X}_{0j}+\widehat{H}_{ij}-\left(\sum \limits_{m=1}^{\infty} \widehat{H}_{im}\right)\boldsymbol{e\pi}_j^{T},\ \ i,j \geq 1.
\end{eqnarray*}
This  completes the proof of Theorem \ref{X-GIG1}.
\end{proof}

\begin{remark} For Markov chains of $GI/M/1$-type, we denote $R_1$ by $R$ for simplicity. From Lemma \ref{R-G-PHI},  the matrices $R$ and $\Psi_0$ satisfy
\[
 R=\Psi_1 (I-\Psi_0)^{-1},\ \ \Psi_0=\sum \limits_{k=0}^{\infty}R^{k}A_{-k},\ \ i\geq 1,
\]
where $\Psi_1=A_{1}$. Moreover, equation  (\ref{rec-pi}) becomes
\begin{eqnarray*}\label{pi}
&&\boldsymbol{\pi}_0^{T} \boldsymbol{e}+\boldsymbol{\pi}_1^{T} (I-R)^{-1} \boldsymbol{e}=1,\\
&&\boldsymbol{\pi}_j^{T}=\boldsymbol{\pi}_1^{T} R^{j-1},\ \ j\geq 1.
\end{eqnarray*}
Thus, the matrix $\widetilde{X}_{0j}, j\geq0$ in Theorem \ref{X-GIG1} is given by
\begin{align*}
\widetilde{X}_{00} & = \left(I-P^{(0)}\right)^{\#}\left(I-\left(I+B_1(I-\Psi_0)^{-1}(I-R)^{-1}\right) \boldsymbol{e\pi}_0^{T}\right), &   \\
\widetilde{X}_{0j} & = \left(I-P^{(0)}\right)^{\#}\left(B_1(I-\Psi_0)^{-1}\left(I-(I-R)^{-1}\boldsymbol{e\pi}_1^{T}\right)- \boldsymbol{e\pi}_1^{T}\right)R^{j-1}, & & j\geq1.
\end{align*}
Please refer to \cite{LWX14} for more details about the matrix solution $\widetilde{X}$ of Markov chains of $GI/M/1$-type.
In particular, for $QBD$ processes, the matrix solution $\widetilde{X}$ is given by Theorem 4.2 in \cite{DLL13}.
\end{remark}

\begin{remark}\label{MG1}
For Markov chains of $M/G/1$-type, we denote $G_1$ by $G$ for simplicity. From Lemma \ref{R-G-PHI}, the matrices $R_{0,i}$ and $R_i$  are given by
\[
R_i=\sum_{n=i}^{\infty}A_nG^{n-i}(I-\Psi_0)^{-1}, \ \ i \geq 1,\]
\[
R_{0,i}  =\sum_{n=i}^{\infty}B_nG^{n-i}(I-\Psi_0)^{-1}, \ \ i\geq1,\]
where $\Psi_0=\sum_{n=0}^{\infty}A_nG^n$. Furthermore, from Lemma \ref{fun-H}, we have
\[\widehat{H}_{j1}=G^{j-1}(I-\Psi_0)^{-1}.\]
\end{remark}

\section{MAP/G/1 Queues with Negative Customers}
\label{sec:MAPG1}
In this section, we give numerical calculations of the matrix $\widetilde{X}$ for  $MAP/G/1$ queues with negative customers.
Queueing systems with negative arrivals have a lot of applications in various areas, such as  computer, manufacturing systems, neural and communication networks, see e.g., \cite{HP96,LZ2004}.

For a single-server FIFO queue, we suppose that there are two types of independent arrivals, positive and negative.
Positive arrivals correspond to customers who upon arrival, join the queue with the intention of being served and then leaving the system.
When a negative customer arrives at the queue, it immediately removes one or more positive  customers if present.
Here, we consider the RCA rule, i.e. arrival of a negative customer which removes all the customers in the system.
Furthermore, we assume that the arrivals of both positive and negative customers are Markovian arrival processes (MAP)
and the service times are independent of the two arrival processes of positive and negative customers and obey a general distribution.
Then the above queueing model is a $MAP/G/1$ queue with negative customers.

In \cite{LZ2004}, Li and Zhao analyzed $MAP/G/1$ queues with negative customers by  introducing  supplementary variables  and constructing the differential equations.
For a  stable RCA system, they related the boundary conditions of the system of differential equations to a  $GI/G/1$ type Markov chain, which is given by
\begin{eqnarray}\label{MAP/G/1}
&&P=\left[
\begin{array}{cccccc}
B_{0}  & B_1     & B_2    & B_3     & B_4& \cdots \\
B_{-1} & A_0     & A_1    & A_2    & A_3 &\cdots \\
B_{-2} & A_{-1}  & A_{0} & A_1    &A_2& \cdots \\
B_{-2} & 0         & A_{-1} & A_0    &A_1 & \cdots\\
B_{-2} & 0         &0          & A_{-1} & A_0 & \cdots\\
\vdots & \vdots & \vdots & \vdots  & \vdots &\ddots
\end{array}
\right].
\end{eqnarray}

It follows from (\ref{H})  that the matrix $H$ of the transition matrix $P$ defined by (\ref{MAP/G/1})  is of the $M/G/1$-type structure.
Thus, we can combine Remark \ref{MG1} and Theorem \ref{X-GIG1} to calculate the matrix solution $\widetilde{X}$ of (\ref{MAP/G/1}). From Remark \ref{MG1}, we know that
the key step is to calculate $G$.
By Proposition 3.5.1 in \cite{H14}, the matrix $G$ can  be computed recursively as follows,
\begin{equation}\label{compute-G}
G[0]=O,\ \ G[k+1]=\sum_{n=-1}^{\infty}A_n(G[k])^{n+1}, \ \ k\geq 0.
\end{equation}
It can be shown that the sequence $\{G[k], \ k\geq 0\}$ is nondecreasing and converges to $G$.
The computation of $G$ is stopped when
\begin{equation}\label{error-G}
\left\|G[k+1]-G[k]\right\|_{\infty}<\varepsilon,
\end{equation}
where $\|M\|_{\infty}=\max_{i}\sum_{j}|M(i,j)|$ denotes the $\infty$-norm of matrix $M$.

Our analysis leads to the algorithm in Algorithm \ref{algorithm} 

\begin{algorithm}\label{algorithm}
Computing the matrix solution $\widetilde{X}$ of (\ref{MAP/G/1}).

\noindent INPUT  the matrices $\{B_i, i\geq -2\}$,  $\{A_i, i\geq -1\}$ and the error $\varepsilon$.

\noindent OUTPUT the matrix solution $\widetilde{X}$.

\noindent COMPUTATIONS:

\noindent Step1: use (\ref{compute-G}) and (\ref{error-G}) to compute $G$.

\noindent Step2: use Remark \ref{MG1} to compute $\Psi_0$, $R_i$ and $R_{0,i}$.

\noindent Step3: use Lemma \ref{fun-H} to compute $\widehat{H}$.

\noindent Step4: use Lemma  \ref{cen-p} to compute $P^{(0)}$ and $\boldsymbol\pi^{(0)}$.

\noindent Step5: use (\ref{compute-(I-P)}) to compute $(I-P)^{\#}$.

\noindent Step6: use (\ref{Pi-A}), (\ref{the constant-c}) and (\ref{rec-pi})  to compute $\boldsymbol\pi$.

\noindent Step6: use Theorem \ref{X-GIG1} to compute $\widetilde{X}$.

\end{algorithm}

\begin{example}\label{example-1}
Consider a  Markov chain of $GI/G/1$-type with transition matrix $P$ given by (\ref{MAP/G/1}). Let $B_i=A_i=0$, $i\geq 3$ and
\[B_0=\left[
\begin{array}{ccc}
0.2 & 0.1 &0.2 \\
0& 0.4&0.1\\
0 &0.2 &0.1
\end{array}
\right], \ \
B_1=\left[
\begin{array}{ccc}
0.2 & 0 &0.2 \\
0.3& 0.1 &0\\
0.4 &0.2 &0
\end{array}
\right],\ \
B_2=\left[
\begin{array}{ccc}
0.1 & 0 &0 \\
0& 0.1 &0\\
0 &0 &0.1
\end{array}
\right],
\]
\[B_{-1}=\left[
\begin{array}{ccc}
0.1 & 0 &0.4 \\
0.3& 0.2 &0.1\\
0 &0.1 &0.3
\end{array}
\right], \ \
B_{-2}=\left[
\begin{array}{ccc}
0 & 0 &0.2 \\
0.1& 0.1 &0\\
0 &0&0.1
\end{array}
\right],\ \ A_{-1}=\left[
\begin{array}{ccc}
0.1 & 0 &0.2 \\
0.2& 0.1 &0.1\\
0.1 &0.1 &0.1
\end{array}
\right]
\]

\[
A_0=\left[
\begin{array}{ccc}
0 & 0.1 &0.2 \\
0.1& 0.1&0\\
0&0.1 &0.2
\end{array}
\right],\ \
A_1=\left[
\begin{array}{ccc}
0.1 & 0 &0 \\
0& 0.1 &0\\
0 &0.1 &0.1
\end{array}
\right],\ \
A_{2}=\left[
\begin{array}{ccc}
0 & 0.1 &0 \\
0& 0.1 &0\\
0 &0 &0.1
\end{array}
\right].
\]
\end{example}

It is obvious that $P$ is irreducible and aperiodic. From Theorem 4.1 in \cite{MTZ12}, we know that the chain is strong ergodic, which implies $\mathbb{E}_{\boldsymbol\pi}\left[\tau_i\right]<\infty$ for every $i\in E$.
Here, we take $\varepsilon=0.0001$.  From (\ref{compute-G})--(\ref{error-G}), we obtain the numerical result of $G$ as follows,
\[G=\left[
\begin{array}{ccc}
0.1802  & 0.0536  &0.2690 \\
0.2610& 0.1268 &0.1610\\
0.1925 &0.1591&0.1813
\end{array}
\right].\]		
From Remark \ref{MG1}, we have
\[\Psi_0=\left[
\begin{array}{ccc}
0.0291  & 0.1109  &0.2389 \\
0.1372& 0.1183 &0.0281\\
0.0565 &0.1345&0.2453
\end{array}
\right],\]		 	 	 	
\[R_1=\left[
\begin{array}{ccc}
0.1398 & 0.0422 &0.0671 \\
0.0492&0.1404&0.0421\\
0.0537 &0.1656 &0.1797
\end{array}
\right], \ \
R_2=\left[
\begin{array}{ccc}
0.0171 & 0.1171 &0.0098\\
0.0171& 0.1171 &0.0098\\
0.0111 &0.0223 &0.1368
\end{array}
\right]
\]  	 	
and
\[R_{0,1}=\left[
\begin{array}{ccc}
0.2609 & 0.0979 &0.3869 \\
0.3722&0.1970&0.1465\\
0.4904  &0.3358 &0.1918
\end{array}
\right], \ \
R_{0,2}=\left[
\begin{array}{ccc}
 0.1077 & 0.0189 &0.0348\\
0.0171& 0.1171 &0.0098\\
0.0111 &0.0223 &0.1368
\end{array}
\right].
\]   	
It follows form Lemma  \ref{cen-p} that
\[
P^{(0)}=\left[
\begin{array}{ccc}
0.2907& 0.1935& 0.5158 \\
0.1389  & 0.4967 &0.3644\\
0.1952  & 0.3318&0.4730 \\
\end{array}
\right],
\] 	 	
from which
\[{(\boldsymbol\pi^{(0)})^T}=(0.1931, 0.3653, 0.4416).\]	 	
By (\ref{rec-pi}), we have
\begin{equation*}
  \boldsymbol{\pi}_n^{T}=c(\boldsymbol{\pi}^{(0)})^{T}R_{0,n}+\sum_{k=1}^{n-1}\boldsymbol{\pi}_k^{T}R_{n-k}, \ \ n\geq1,
\end{equation*}	
where $c$ is a constant such that $c=\pi_0(0)+\pi_0(1)+\pi_0(2)$.  From (\ref{the constant-c}), we obtain that $c=0.3563$ and
the invariant probability vector which is given as follows,
\begin{table}[h]
\caption{ The invariant probability vector of Example \ref{example-1}.}
  \centering
  \begin{tabular}{cc|cc|cc|cc|cc}
    \hline
        \hline
    $state$ & $value$  & $state$ & $value$ & $state$ & $value$& $state$ & $value$& $state$ & $value$\\
\hline
    (0,1) & 0.0688 & (3,1)&0.0156&(6,1)& 0.0025 &(9,1)&0.0004&(12,1)&0.0001\\
    (0,2) &0.1302  & (3,2)&0.0459&(6,2)&0.0076&(9,2)&0.0014&(12,2)&0.0002\\
    (0,3) & 0.1574& (3,3)&0.0268  & (6,3) & 0.0049 &(9,3)&0.0009& (12,3) &	0.0002 \\
 \hline
     (1,1) &0.1436&(4,1)&0.0080&(7,1)&0.0014&(10,1)&0.0002 &(13,1)&0.0000\\
   (1,2) & 0.0852&(4,2)&0.0233&(7,2)&0.0043&(10,2)&0.0008&(13,2)&0.0001\\
    (1,3) & 0.0759&(4,3)&0.0158&(7,3)&0.0027&(10,3)&	0.0005 &(13,3)&0.0001\\
\hline
 (2,1) & 0.0397&(5,1)&0.0045&(8,1)&0.0008&(11,1)&	0.0001&(14,1)&	0.0000 \\
   (2,2) &0.0506&(5,2)&0.0140&(8,2)&0.0024&(11,2)&	0.0004&(14,2)&0.0001\\
    (2,3) & 0.0521&(5,3)&0.0086&(8,3)&0.0015&(11,3)&   	0.0003&(14,3)&0.0000\\
    \hline
        \hline
  \end{tabular}
\end{table}	 	

From Theorem \ref{X-GIG1}, we obtain the matrix solution $\widetilde{X}$ 	as follows, 	 	 	 	 	 	 	  	 	 	 	 	 	 	
\begin{equation*}
\widetilde{X}=
\left[\begin{array}{ccc|ccc|ccc|c}
0.931  & -0.587 & -0.333 & -0.139 &-0.136&0.210 & 0.072 &-0.049& 0.004& \cdots \\
-0.237 & 0.854  & -0.480 & 0.021  &-0.016&-0.059 &-0.017 & 0.074 &-0.075&\cdots \\
-0.211 & -0.449 & 0.542  & 0.044  &0.073& -0.044 &-0.017 &-0.040&0.061& \cdots \\
\hline
-0.160  & -0.684& -0.227 & 0.669& -0.038  & 0.142& 0.068 &	-0.042&0.026& \cdots \\
0.100 &-0.431  &-0.516 &-0.216  &0.922 & 	-0.026  & 0.011&0.056 &-0.052& \cdots \\
-0.283 & -0.585 &-0.376 &-0.365 &-0.046 & 1.101 &-0.039 &0.078&0.104& \cdots \\
\hline
-0.318 & -0.812  & -0.487 & -0.391 & -0.198 & 0.100 & 0.964&0.049&0.246& \cdots \\
-0.142 & -0.642 &	-0.649 &	-0.274 	&-0.125 &	0.050 	&0.096 &	1.057 	&-0.017& \cdots \\
-0.296& 	-0.851& 	-0.648& 	-0.457 	&-0.151 	&-0.033 &	-0.016&	0.057 &	1.218& \cdots \\
\hline
\vdots & \vdots & \vdots & \vdots & \vdots & \vdots & \vdots &\vdots&\vdots& \ddots
\end{array}\right].
\end{equation*}

Now, we take $g(i,j)=i\times j$ in this model. By our calculations, we have $\boldsymbol\pi^T|\boldsymbol{g}|=2.9356<\infty$.
Thus, the solutions $\boldsymbol{f_D}$ and ${\boldsymbol{f_K}}$ exist simultaneously.
Taking $\alpha=(0,1)$, from Corollary \ref{expression-D} and Theorem \ref{result-expression-add-functional}, we obtain the values of $\boldsymbol{f_D}$ and ${\boldsymbol{f_K}}$,
and the values of those solutions are plotted in Figure \ref{Fig1}. From Remark \ref{f_d,f_k}, we know that $\boldsymbol{f_D}(i)-{\boldsymbol{f_K}}(i)= -10.3247$ for every $i\in E$.
Note that the $x$-axis represents the state space, in which the origin is the state $(0, 1)$,
and $n$ represents state $(i, j)$ such that $n = 3i + j-1$, $0\leq j\leq2$.
\begin{figure}[h]
  \centering
  \includegraphics[width=8cm]{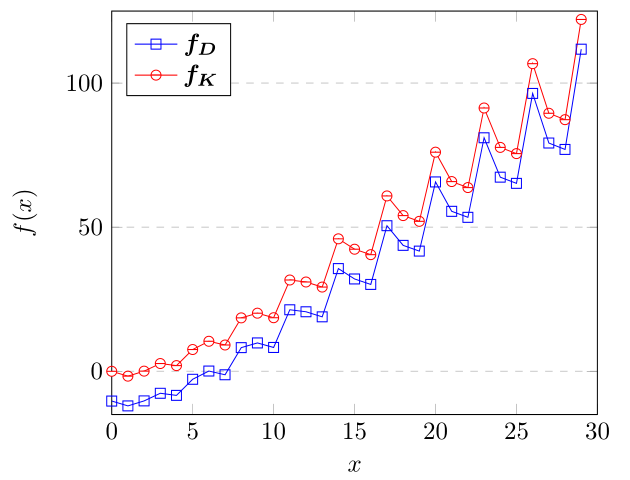}
  \caption{The values of the solutions $\boldsymbol{f_D}$ and $\boldsymbol{f_K}$ with $\alpha=(0,1)$ of Example \ref{example-1}. }\label{Fig1}
\end{figure}

For a scalar-valued  Markov chain of GI/G/1-type, we can present the  analytic expression of the matrix solution $\widetilde{X}$.
\begin{example}\label{example-2-0}
Consider a DTMC with the following stochastic transition matrix:
\begin{equation}\label{example-2}
 P=\left[
 \begin{array}{ccccccc}
   b_{0}   & b_{1} & b_{2}     & b_{3}  & b_{4}  & \cdots\\
   b_{-1}   & a_{1}    & a_{2} & a_{3}  & a_{4}  & \cdots\\
   b_{-2}   &  a_{0}     &a_{1}      & a_{2} & a_{3} & \cdots\\
   b_{-2}   & 0     & a_{0}      &a_{1}  & a_{2}   & \cdots\\
   b_{-2}   & 0    & 0  & a_{0}      &a_{1}   & \cdots\\
   \vdots&\vdots &\vdots& \vdots& \vdots & \ddots \\
   \end{array}\right],
\end{equation}
where $0<b_{-1}, b_{-2}<1$.
\end{example}

Clearly P is irreducible and aperiodic. From Theorem 16.0.2 in \cite{MT09}, we know that the chain is strong ergodic, which implies $\mathbb{E}_{\boldsymbol\pi}\left[\tau_i\right]<\infty$ for every $i\in E$.
Now, let $\mathcal{G}$ be the minimal nonnegative solution of equation  $\mathcal{G}=\sum_{i=0}^{\infty}a_i\mathcal{G}^i$. Moreover, we let
\[\psi_0=\sum_{i=0}^{\infty}a_{i+1}\mathcal{G}^i\]
\[
r_n=(1-\psi_0)^{-1}\sum_{i=1}^{\infty}a_{i+n}\mathcal{G}^{i-1}, \ \ n\geq 1,\]
and
\begin{equation*}
\mathcal{H}=\left[
 \begin{array}{ccccc}
  a_{1}    & a_{2} & a_{3}  & a_{4}  & \cdots\\
    a_{0}     &a_{1}      & a_{2} & a_{3} & \cdots\\
   0     & a_{0}      &a_{1}  & a_{2}   & \cdots\\
  0    & 0  & a_{0}      &a_{1}   & \cdots\\
\vdots &\vdots& \vdots& \vdots & \ddots \\
   \end{array}\right].
\end{equation*}

In (\ref{example-2}), let $a_i=2^{-i-2}$, $i\geq 0$; $b_i=2^{-i-1}$, $\geq0$; $b_{-1}=\frac{3}{4}$ and $b_{-2}=\frac{1}{2}$.
Here we take $A=\{0,1\}$. Then, we have
\begin{eqnarray}
\left.
\begin{array}{cc}
P_{AA}=\left[
\begin{array}{cccc}
1/2& {1}/{4}\\
{3}/{4} & {1}/{8}
\end{array}
\right], & P_{AB}=\left[
\begin{array}{cccc}
{1}/{8}&{1}/{16}&{1}/{32}&\cdots\\
{1}/{16}&{1}/{32}&{1}/{64}&\cdots
\end{array}
\right], \\
\\
 P_{BA}=\left[
\begin{array}{cccc}
{1}/{2}&{1}/{4}\\
{1}/{2}&0\\
{1}/{2}&0\\
  \vdots&\vdots
\end{array}
\right],   & P_{BB}=\mathcal{H}=\left[
\begin{array}{cccc}
 1/8 & 1/16 &1/32 &\cdots \\
 1/4 &1/8 &1/16& \cdots \\
0& 1/4 & 1/8& \cdots\\
  \vdots & \vdots & \vdots &\ddots
\end{array}
\right]. \\
\end{array}
\right.
\end{eqnarray}
By calculations, we obtain
\[\mathcal{G}=\frac{2-\sqrt{2}}{2},\ \ \psi_0=\frac{2-\sqrt{2}}{4},\ \ r_n=\frac{(2-\sqrt{2})^2}{2^{n+1}}.\]
From Lemma \ref{fun-H}, we have
\[\widehat{\mathcal{H}}(1,1)=2(2-\sqrt{2}); \ \  \widehat{\mathcal{H}}(i,i)=\widehat{\mathcal{H}}(1,1)+\mathcal{G}\widehat{\mathcal{H}}(i-1,i),  \ i\geq 2;
\ \  \widehat{\mathcal{H}}(i,j)=\mathcal{G}^{i-j}\widehat{\mathcal{H}}(j,j), \ i> j;\]
\[\widehat{\mathcal{H}}(i,i+1)=r_1\sum_{n=1}^{i}\left(\frac{1}{2}\right)^{i-n}\widehat{\mathcal{H}}(i,n);\ \  \widehat{\mathcal{H}}(i,j)=(2-\sqrt{2})^{j-i-1}\widehat{\mathcal{H}}(i,i+1),\ j-i>1. \]
It follows form Lemma  \ref{cen-p} that
\[
P^{(A)}=\left[
\begin{array}{cc}
0.7071& 0.2929  \\
0.8536 & 0.1464
\end{array}
\right],
\]
from which, we have
\[({\boldsymbol{\pi}^{(A)}})^T=(0.7445, 0.2555).\]
By (\ref{Pi-A}), (\ref{the constant-c}) and (\ref{pi-1}), we obtain the invariant probability vector which is given as follows,
\begin{table}[h]
\caption{The invariant probability vector of Example \ref{example-2-0}.}
  \centering
  \begin{tabular}{cccccccccc}
    \hline    \hline
    $state$ & 0 & 1&2&3&4&5&6&7&8\\
    \hline
$value$    & 0.5469 & 0.1877& 0.1099 & 0.0644 & 0.0377 & 0.0221 & 0.0129 & 0.0076 & 0.0044\\
 \hline
    $state$   &  9& 10&11&12&13&14&15&16&17\\
    \hline
$value$      & 0.0026   &0.0015 &0.0009&0.0005 & 0.0003 & 0.0002 & 0.0001 & 0.0001 & 0.0000  \\
    \hline    \hline
  \end{tabular}
\end{table}

It follows form Theorem \ref{the-result-general-MC} that
\begin{equation*}
\widetilde{X}=\left[
 \begin{array}{cccccccccc}
0.198     & -0.231 & 0.014  & 0.008  &0.005&0.003 &	0.002 &0.001&0.001& \cdots\\
-0.576    & 0.675 & -0.041  &-0.024  &-0.014&-0.008 &-0.005&0.003&-0.002& \cdots\\
-0.802    &-0.231 & 1.014  & 0.008  &0.005&0.003  &	0.002 &0.001&0.001 & \cdots\\
-0.869  &	-0.497 	 &0.151  &	1.089 	 &0.052  &	0.030 & 	0.018  &	0.010 & 	0.006 & \cdots\\
-0.888  &	-0.575 & 	-0.101 & 	0.234 	 &1.137  &	0.080 & 	0.047  &	0.028  &	0.016 	& \cdots\\
-0.894  &	-0.597 	 &-0.175  &	-0.017 & 	0.283 & 	1.166 & 	0.097  &	0.057 & 	0.033 	& \cdots\\
-0.896 & 	-0.604  &	-0.197 	 &-0.090 & 	0.033 & 	0.312 	 &1.183 & 	0.107  &	0.063 	& \cdots\\
-0.896  &	-0.606 & 	-0.203  &	-0.112 & 	-0.040 & 	0.062 & 	0.329  &	1.193 	 &0.113 	& \cdots\\
-0.896 & 	-0.607 & 	-0.205  &	-0.118  &	-0.062 & 	-0.011  &	0.079 & 	0.339 & 	1.199 	& \cdots\\
\vdots &\vdots& \vdots& \vdots &\vdots &\vdots &\vdots& \vdots& \vdots & \ddots \\
   \end{array}\right].
\end{equation*}

Letting $g(i)=\sqrt{i}$, we get $\boldsymbol{\pi}^T|\boldsymbol{g}|=0.6641<\infty$ and the solutions $\boldsymbol{f_D}$, $\boldsymbol{f_K}$ exist simultaneously.
Taking 	 $\alpha=0$,  from Corollary \ref{expression-D} and Theorem \ref{result-expression-add-functional},
we obtain the solutions $\boldsymbol{f_D}$ and $\boldsymbol{f_K}$  and the values of those solutions are plotted in Figure \ref{Fig2}.
From Remark \ref{f_d,f_k}, we know that $\boldsymbol{f_D}(i)-{\boldsymbol{f_K}}(i)= -0.7079$ for every $i\in E$.

\begin{figure}[h]
  \centering
  \includegraphics[width=8cm]{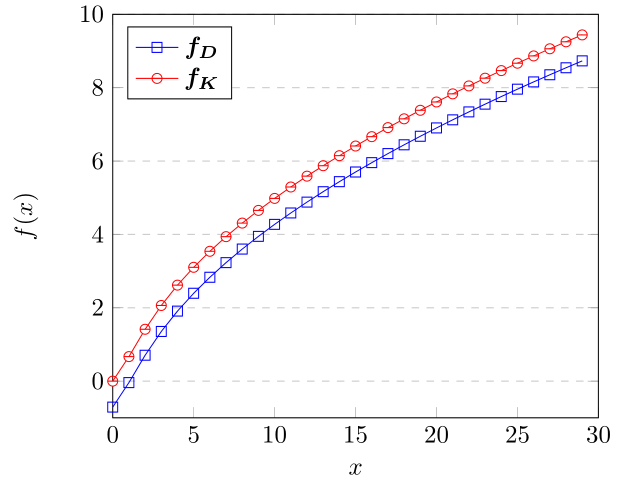}
  \caption{The values of the solutions $\boldsymbol{f_D}$ and $\boldsymbol{f_K}$ with $\alpha=0$ for Example \ref{example-2-0}. }\label{Fig2}
\end{figure}

\section{Continuous-time Markov chains}
\label{sec:CTMC}

It is of the same feasibility to investigate continuous-time Markov chains (CTMCs) by using matrix-analytic methods.
Let $Q = (Q(i, j))_{i,j\in E}$ be a totally stable and regular generator of the CTMC $\mathbf{\Phi}=\{\Phi_t, t\geq 0\}$ on a countable state space $E$.
It is assumed that $\mathbf{\Phi}$  is irreducible and positive recurrent with the unique invariant probability vector ${\boldsymbol\pi}$.
For a given $Q$, Poisson's equation  is written as
\begin{equation}\label{poi-fun-2}
-QX=I-\boldsymbol{e\pi}^T.
\end{equation}

For a non-empty subset $A$ of $E$, let $\mathbf{\Phi}^{(A)}=\{\Phi^{(A)}_t, t\geq0\}$ be the censored Markov chain on  $A$ with the generator $Q^{(A)}$.
According to section 5 in \cite{GV99}, the generator $Q^{(A)}$  is given by
\begin{equation}\label{gen-tru-2}
Q^{(A)}=Q_{AA}+Q_{AB}\widehat{Q}_{BB}Q_{BA},
\end{equation}
where $B$ is the complement of set $A$ , and $\widehat{Q}_{BB}$, defined by
\[\widehat{Q}_{BB}:=\int_{0}^{\infty}\exp(Q_{BB}t)dt,\]
is the minimal nonnegative solution of
\[X(-Q_{BB})=(-Q_{BB})X=I.\]

Using the similar arguments in the proof of Theorem \ref{the-result-general-MC} leads the following results. The proof will be omitted.

\begin{theorem}\label{the-result-CTMC}
Let $\mathbf{\Phi}$ be an irreducible and positive recurrent CTMC and let $A$ be a finite subset of $E$.
Then, the matrix $\widetilde{X}$, given by
\begin{equation}\label{the-matrix-K-A-2}
\widetilde{X}_{A}=\left(-{Q}^{(A)}\right)^{\#}\left(\left[I, Q_{AB}\widehat{Q}_{BB}\right]-\left(\boldsymbol e+Q_{AB}\widehat{Q}_{BB}\boldsymbol e\right)\boldsymbol\pi^T\right)
\end{equation}
and
\begin{equation}\label{the-matrix-K-B-2}
\widetilde{X}_{B}=\widehat{Q}_{BB}Q_{BA}\widetilde{X}_A+\left[O,\widehat{Q}_{BB}\right]-\widehat{Q}_{BB}\boldsymbol e \boldsymbol \pi^T,
\end{equation}
is one solution of  Poisson's equation (\ref{poi-fun-2}).
\end{theorem}

\begin{remark}
For CTMCs, we can also define the deviation matrix $D$ and the expected integrable-type functional matrix $K$, see e.g., [7, 11].
By using Theorem \ref{the-result-CTMC}, we can obtain $D$ and $K$ in term of similar arguments in Corollary \ref{expression-D} and Theorem \ref{result-expression-add-functional},  respectively.
For CTMCs of GI/G/1type, we can also obtain parallel results to  that in sections 4 and 5.
\end{remark}

\section{Concluding remarks}
\label{sec:conclusions}
In the previous sections, we have investigated the  matrix solution $\widetilde{X}$ of Poisson's equation for general DTMCs by developing matrix-analytic methods.
Interestingly, we obtain the connection between the  matrix solution $\widetilde{X}$ and the matrix solution $\left(I-P^{(A)}\right)^{\#}$
of Poisson's equation for the censored Markov chain $\mathbf{\Phi}^{(A)}$ in the process of solving the matrix solution $K$.
Furthermore, we derive an explicit expression  of the matrix solution $\widetilde{X}$ for  Markov chains of $GI/G/1$-type,
which includes  results of  $QBD$ processes,  Markov chains of $GI/M/1$-type and  Markov chains of $M/G/1$-type.

There are other areas in which one might extend our studies.
The first possible extension is to consider level dependent Markov chains of $GI/G/1$-type.
It can be expected that the calculations of the $R$-measures and $G$-measures will become more challenging for the level dependent case.
The other possible extension is to consider Poisson's equation for positive recurrent fluid queues.
The arguments in this paper and Soares and Latouche \cite{SL06} may be modified to use,
but evidently it requires essentially different arguments to deal with the case of non-countable state space.

\section*{Acknowledgements}
This work were funded by the National Natural Science Foundation of China (Grants No. 11971486), Natural Science Foundation of Hunan (Grants No. 2020JJ4674) and Discovery Grant of the
Natural Sciences and Engineering Research Council of Canada (Grants No. 315660).
\normalem
\bibliography{mybibfile}
\bibliographystyle{unsrt}

\end{document}